\documentclass[12pt,a4paper]{amsart} %jesli chce numeracje jednostronnq, to daje ,oneside, w graniaste.
\usepackage{amsmath}
\usepackage{amsthm}
\usepackage{amssymb}
\usepackage{ot1patch}

\newtheorem{thm}{Theorem}[section]
\newtheorem{lem}[thm]{Lemma}
\newtheorem{cor}[thm]{Corollary}
\newtheorem{prop}[thm]{Proposition}

\theoremstyle{remark}
\newtheorem{rem}[thm]{Remark}
\newtheorem*{rem*}{Remark}

\theoremstyle{definition}

\numberwithin{equation}{section}

\newcommand{\Rz}{\mathbb{R}}

\begin{document}
\title[Yomdin's version of a Lipschitz Implicit Function Theorem\dots]{On Yomdin's version of a Lipschitz Implicit Function Theorem and the geometry of medial axes}

\author{Maciej P. Denkowski}\address{Jagiellonian University, Faculty of Mathematics and Computer Science, Institute of Mathematics, \L ojasiewicza 6, 30-348 Krak\'ow, Poland}\email{maciej.denkowski@uj.edu.pl}\date{January 23rd 2014 for the first, incomplete version, updated November 28th 2023 and March 19th 2024}
\keywords{Medial axis, central set, Clarke's subgradient, Lipschitz Inverse Function Theorem.}%, tame geometry, subanalytic sets.}
\subjclass{32B20, 26A27}

\begin{abstract}
In his beautiful paper on the central set from 1981, Y.~Yom\-din makes use of a Lipschitz Inverse Function Theorem that seemingly has been unproved until now. After a brief discussion of a natural and straightforward Lipschitz counterpart of an implicit function theorem, based on a geometric condition we finally provide a proof of Yomdin's version holds by proving the geometric condition is in fact equivalent to the one given by Yomdin. Therefore, Yomdin's Generic Structure Theorem, whose updated version is also presented here, concerning the medial axis (central set) of a subset of %definable or subanalytic
in ${\Rz}^n$ is now flawless. We also note that Yomdin's Lipschitz Implicit Function Theorem is equivalent to Clarke's Lipschitz Inverse Function Theorem. The paper ends with some additional properties of Lipschitz germs satisfying the Yomdin condition (e.g. a Lipschitz triviality result).
\end{abstract}
\maketitle

\section{Introduction}
The present paper has a long story of misdirection and quid pro quos, which is a rather unusual instance for a mathematical paper. Its main goal is to give a correct proof of Yomdin's version of a Lipschitz Implicit Function Theorem specialists from Lipschitz Geometry have long been waiting for. As it happens, the proof is strangely a very simple one.

Let us start with the background story. In \cite{Y} Yomdin considers the following situation: $X\subset\mathbb{R}^n$ is nonempty and closed, $a\in \mathbb{R}^n$ is such a point that the compact set of closest points computed according to the Euclidean norm $$m(a):=\{x\in X\mid ||a-x||=\mathrm{dist}(a,M)\}$$ consists of at least two disjoint closed sets. Fix any $k\geq 2$ such that $m(a)$ is the union of $k$ disjoint closed nonempty sets $X_j$ (in particular $k$ does not exceed the number of connected components of $m(a)$). Then we choose disjoint neighbourhoods $V_j\supset \overline{W_j}\supset X_j$ and put $\delta_j(x)=\mathrm{dist}(x,X\cap \overline{W_j})^2$. The germ of $\delta_j$ depends only on $X_j$ and in some neighbourhood $U\ni a$ there is $\delta(x)=\min\delta_j(x)$. 

Recall the so-called \textit{medial axis} (self-conflict set or exceptional set) of $X$ (see \cite{BD} for the point of view of singularity theory, or \cite{D} with another notation, but also \cite{C})
$$
M_X=\{x\in{\Rz}^n\mid \#m(x)>1\}.
$$
This set plays an major role e.g. in pattern recognition since its introduction by H. Blum in the late 1960s.

Let $C(a):=\{x\in U\mid \exists i\neq j\colon \delta_i(x)=\delta_j(x)=\delta(x)\}$ (this is a part of the conflict set of the sets $X\cap\overline{W_j}$). Then $C(a)\subset M_X$ and $U\cap M_X\setminus C(a)=\bigcup_1^k C_j(a)$ for some sets $C_j(a)\subset M_{X\cap \overline{W_j}}$ (a kind of self-conflict sets) --- cf. \cite{Y}. In some sense, $C(a)$ may be seen as the `core' of $M_X$ at $a$.

\begin{thm}[\cite{Y} Theorem 1, Yomdin's Generic Structure Theorem]\label{YS}
In the setting introduced above, assume that $k\leq n+1$ and for any choice of points $y_j$ in the convex hulls $\mathrm{cvx}(X_j)$, the system $y_1,\dots, y_k$ defines a $(k-1)$-dimensional simplex in $\mathbb{R}^n$. 

Then there is a neighbourhood $V$ of $a$ and a bi-Lipschitz homeomorphism $\phi\colon (V,a)\to (\Delta^{k-1}\times (-1,1)^{n-k+1}, (c,0))$ where $\Delta^{k-1}$ is the open $(k-1)$-dimensional simplex with centre $c$, and we have $\phi(C(a)\cap V)=M^{(k)}\times (-1,1)^{n-k+1}$ where $M^{(k)}:=M_{\partial \Delta^{k-1}}\cap\Delta^{k-1}$.
\end{thm}

The problem with this result is that Yomdin's proof is based on a kind of Implicit Function Theorem for Lipschitz functions (LIFT) seemingly being a mere corollary to Clarke's Inverse Function Theorem for Lipschitz functions \cite{C2}. It can be stated as follows:

\begin{thm}[Yomdin's LIFT]\label{LIFT}
Let $f\colon ({\Rz}^n, 0)\to ({\Rz}^k,0)$ be a Lipschitz germ with $k<n$ and such that any linear map $\ell$ from the Clarke's subdifferential $\partial f(0)$ has maximal rank $k$. Then, after a linear change of coordinates, the germ $(f^{-1}(0),0)$ is the graph of a Lipschitz function of the first $n-k$ coordinates.
\end{thm}

Unfortunately, there is apparently no such Lipschitz Implicit Function Theorem proved anywhere and the version cited in \cite{Y} for a long time has remained unproved, even though of utmost interest for people working in Lipschitz geometry of singularities. By the main result of \cite{BS} it was known that the Generic Structure Theorem above was true for $k=2$, but the example from \cite{BS} Section 3 with $n=3$ seemed to imply a possible counter-example to the general case. And as a matter of fact, without a closer examination of this example, one could have been easily mislead, which happened in particular to the author of this paper.

In a preprint prepared in 2014 and published on arxiv in 2016 \cite{Dp} we gave a possible correct version of LIFT and applied it to obtain a proper formulation of Yomdin's Structure Theorem. We also showed that Yomdin's version of LIFT is true for scalar functions. Having in mind the erroneous {\it id\'ee fixe} that the example from \cite{BS} contradicts Yomdin's Structure Theorem, we have been searching for a counter-example to Yomdin's version of LIFT for quite a few years. Obviously in vain, as proves the present paper.

\section{Lipschitz Implicit Function Theorem}

Recall that by the Rademacher Theorem any locally Lipschitz function $f\colon\mathbb{R}^n\to\mathbb{R}$ has a well defined (and locally bounded) gradient almost everywhere: the set $D_f$ of differentiability points is dense. We define according to \cite{C} the \textit{generalized gradient at} $x\in\mathbb{R}^n$ as the convex hull of the set of all possible limits $\lim\mathrm{grad} f(x_\nu)$ when $D_f\ni x_\nu\to x$. We denote the resulting nonempty convex compact set by $\partial f(x)$. The set $\partial f(x)$ is reduced to a singleton (being the gradient of $f$ at $x$) iff the function $f$ is differentiable at $x$  and $\mathrm{grad} f|_{D_f}$ is continuous at $x$. Of course, the same arguments apply to a vector-valued function. More can be found in \cite{C}. 

As stated in the introdcution, in \cite{Y} Yomdin uses an apparently unknown version of the Lipschitz Implicit Function Theorem (LIFT). The basic setting is the following: we consider a Lipschitz germ $f\colon (\Rz^n,0)\to (\Rz^k, 0)$ with $k<n$ so that the Clarke subdifferential $K:=\partial f(0)$ is well-defined (and a compact, convex subset of $({\Rz}^n)^k$). Recall that this is the convex hull of the limits at $0$ of Jacobian matrices at nearby differentiable points (by the Rademacher Theorem those points are dense in the domain of $f$). In particular, $K\subset \partial f_1(0)\times\ldots\times \partial f_k(0)$ where $f=(f_1,\dots, f_k)$. We identify the matrix $\ell\in K$ with the linear map it defines.

The general aim is to show that under some condition on $K$ the level set germ $(f^{-1}(0),0)$ is the graph of a Lipschitz function of $n-k$ variables (then we say that {\it LIFT holds}). As we have seen in the introduction, Yomdin invokes the following condition:
$$
\forall \ell\in K, \mathrm{rk}\ell=k,\leqno{(Y)}
$$
but there is actually no reference for its sufficiency. As a matter of fact, the misunderstanding comes probably from the fact that in classical analysis the Implicit Function Theorem is actually equivalent to the Local Diffeomorphism Theorem and the latter has its Lipschitz counterpart proved by Clarke in \cite{C2} under the condition $(Y)$ and, of course, in the case $k=n$:

\begin{thm}[Clarke, \cite{C2}]\label{CLIFT}
    Let $F\colon ({\Rz}^n, 0)\to ({\Rz}^n,0)$ be a Lipschitz germ such that any linear map $\ell$ from $\partial F(0)$ has the maximal rank $n$. Then $F$ is the germ of a bi-Lipschitz mapping.
\end{thm}

However, there is no straightforward, simple way of obtaining the desired LIFT from this result. When trying to adapt the classical proof, one finds immediately that a sufficient condition is the following one (here $G_k(\Rz^n)$ denotes the $k$-th Grassmannian):
$$
\exists \Lambda\in G_k(\Rz^n)\colon \forall \ell\in K, \Lambda\cap \mathrm{Ker}\ell=\{0\}.\leqno{(*)}
$$
\begin{rem}\label{implikacja}
    Note that $\dim \mathrm{Ker}\ell\geq n-k$ while $\dim \Lambda\cap \mathrm{Ker}\ell\geq k+\dim \mathrm{Ker}\ell-n$, whence $(*)$ implies $\dim\mathrm{Ker}\ell=n-k$, i.e. $(Y)$ follows from $(*)$. 
    \end{rem}

This direct approach was known to Clarke. In fact, the following Theorem can be found in \cite{Clarke Nonsmooth} section 7.1, although with a different wording, much less geometric and thus less convenient for our purposes. Clarke considers a locally Lipschitz function $f\colon {\Rz}^{m}\times\Rz^{k}\to{\Rz}^k$ in the variables $(x,y)$ and defines a kind of partial subdifferential denoted by $\pi_y\partial f(x,y)$ and consisting of all $k\times k$ matrices $M$ for which there exists a $k\times m$ matrix $N$ such that the $k\times (m+k)$ matrix $[N,M]$ belongs to $\partial f(x,y)$. Keeping the notation introduced so far, if we write $\Rz^n=\Rz^m\times\Rz^k$, Clarke's condition reads
$$
\forall M\in\pi_y K, \mathrm{rk} M=k. \leqno{(C)}
$$
Clearly, this is equivalent to $(*)$ after a convenient linear change of coordinates.

We will give here an elementary proof of this seemingly weaker version of LIFT under condition $(*)$.
\begin{thm}\label{main}
    Condition $(*)$ is enough for LIFT to hold, i.e. the final assertion of Theorem \ref{LIFT} is true when condition $(Y)$ is replaced by $(*)$.
\end{thm}
\begin{proof}
    After a linear change of coordinates we may assume that the subspace from condition $(*)$ is $\Lambda=\{0\}^{n-k}\times{\Rz}^k$. Then $F(x,y):=(x,f(x,y))$, for $(x,y)\in{\Rz}^{n-k}\times{\Rz}^k$, is Lipschitz and for any $L\in\partial F(0)$, there is $L\in \partial \mathrm{id}_{\Rz^{n-k}}(0)\times\partial f(0)$, i.e. $L(x,y)=(x,{\ell}(x,y))$, for some $\ell\in K$. Then $(*)$ implies that $\mathrm{rk}L=n$ and so by the Clarke Lipschitz Inverse Function Theorem \ref{CLIFT}, $F$ is bi-Lipschitz at the origin. The inverse must have the form $(x,y)\mapsto (x,h(x,y))$ with $h$ Lipschitz. Then $f^{-1}(0)$ is the graph of $x\mapsto h(x,0)$.
\end{proof}

As observed for the first time in \cite{Dp}, for $k=1$ we can obtain LIFT under the assumption $(Y)$ as it follows from the next result.
\begin{prop}
For $k=1$, $(Y)$ implies $(*)$.
\end{prop}
\begin{proof}
    The set $K$ is compact, convex and $0\notin K$, whence by the Hahn-Banach Theorem we can find a hyperplane $H\subset{\Rz}^n$ such that $H\cap K=\varnothing$. Then $\Lambda:=H^\bot$ is a line and for $\ell\in K$, we have
    $$
    \mathrm{Ker}\ell\cap \Lambda\neq\{0\}\ \Leftrightarrow\ \mathrm{Ker}\ell\supset \Lambda\ \Leftrightarrow\ (\mathrm{Ker}\ell)^\bot\subset H,
    $$
    but as $(\mathrm{Ker}\ell)^\bot=\mathbb{R}\ell$, the above is equivalent to $\ell\in H$. Therefore, $(*)$ holds for $\Lambda$.
\end{proof}
%\begin{rem}
 %   Note that in the proof above we cannot just choose any line complementary to $H$ -- actually it cannot lie inside the full cone spanned over $K$.
%\end{rem}

This suggests to approach the problem by induction on $k$. Therefore, we will write $f=(f_1,g)$ with $g=(f_2,\dots, f_k)$. In order to control the dimensions in the right way we will have to consider the restriction of $g$ to some hyperplane $H\subset\Rz^n$ that would ensure that the resulting subdifferential consists only of elements of maximal rank. However, there arises the problem of the relation between $\partial (g|_H)$ and $\partial f(0)$. This can be avoided once we observe that the problem we are investigating is in fact a question concerning a convex, compact set $K$ of matrices of maximal rank. We may thus forget the function $f$. 

\begin{thm}\label{Glowne}
    For any $k$, the conditions $(Y)$ and $(*)$ are equivalent.
\end{thm}
\begin{proof}
In order to prove this result, we have to show that $(Y)$ implies $(*)$ (cf. Remark \ref{implikacja}). 
We have thus a nonempty, compact, convex set $K$ of $k\times n$ matrices with linearly independent rows as vectors from $\Rz^n$. Now, a natural idea is to consider $\pi(\ell_1,\dots, \ell_k)=\ell_1$, $(\ell_1,\dots, \ell_k)\in (\Rz^n)^k$ and take $K':=\pi(K)$ which is a compact, convex set with $0\notin K'$. Again the Hahn-Banach Theorem gives us a hyperplane $H\subset \Rz^n$ disjoint with $K'$. For any $(\ell_1,\dots, \ell_k)\in K$, the space $\mathrm{Span}\{\ell_1, \dots, \ell_k\}$ is either contained in $H$, or the intersection with $H$ lowers its dimension by 1. Since we have $\Rz \ell_1\cap H=\{0\}$, it follows $\dim\mathrm{Span}\{\ell_1, \dots, \ell_k\}\cap H=k-1$ and thus $\ell_i|_H$, $i=2,\dots, k$ are linearly independent. 

For $(\ell_1,\dots, \ell_k)\in K$, write $\tilde{\ell}:=(\ell_2,\dots, \ell_k)$. We have shown that $\tilde{\ell}|_H\colon H\to {\Rz}^{k-1}$ has rank $k-1$. This is true for any $\tilde{\ell}\in \tilde{K}$, where $\tilde{K}=\tilde{\pi}(K)$, for $\tilde{\pi}(\ell_1,\tilde{\ell})=\tilde{\ell}$. This means $(Y)$ holds for $\{\tilde{\ell}|_H\mid \tilde\ell\in\tilde{K}\}$ and thus, by the induction hypothesis, $(*)$ is also true, i.e. there is $\Lambda_0\in G_{k-1}(H)$ such that $$\forall\tilde{\ell}\in \tilde{K}, \Lambda_0\cap \mathrm{Ker}\tilde{\ell}|_H=\{0\}.$$ Of course, formally, we should have taken also an isomorphism $\phi\colon{\Rz}^{n-1}\to H$ and replace the restrictions $\tilde{\ell}|_H$ by $\tilde{\ell}\circ\phi$, i.e. the elements of $\phi^*\tilde{K}$, but we may skip this step in the argument.

Note that since $\Lambda_0\subset H$ and $\mathrm{Ker}\tilde{\ell}|_H=\mathrm{Ker}\tilde{\ell}\cap H$, it follows that we have $\Lambda_0\cap \mathrm{Ker}\tilde{\ell}=\{0\}$ whence $\Lambda_0\cap \mathrm{Ker}\tilde{\ell}\cap\mathrm{Ker}\ell_1=\{0\}$. The problem now is that we still need to extend $\Lambda_0$ to the right dimension and keep track of the intersection with $\mathrm{Ker}\ell_1$.

In order to avoid the problem of expanding $\Lambda_0$ we may adopt yet another point of view. First observe that given an epimorphism $\ell=(\ell_1,\dots,\ell_k)$, there is $(\mathrm{Ker}\ell)^\bot=\mathrm{Span}\{\ell_1,\dots, \ell_k\}$. Therefore, we have the following observation:
\begin{lem}
  For an epimorphism $\ell=(\ell_1,\dots,\ell_k)\colon {\Rz}^n\to{\Rz}^k$ and $\Lambda\in G_k({\Rz}^n)$, there is 
  $$
  \Lambda\cap\mathrm{Ker}\ell=\{0\}\ \Leftrightarrow\ \mathrm{Span}\{\ell_1,\dots, \ell_k\}\cap\Lambda^\bot=\{0\}.
  $$
\end{lem}
\begin{proof}[Proof of the Lemma]
    Take a unitary isomorphism $\Phi\colon {\Rz}^n\to{\Rz}^n$ sending the $k$-dimensional space $\Lambda$ onto $\mathrm{Span}\{\ell_1,\dots, \ell_k\}$ and $\Lambda^\bot$ onto $\mathrm{Ker}\ell$. Then 
    $$
    \{0\}=\Lambda\cap\mathrm{Ker}\ell=\Lambda\cap\Phi^{-1}(\Lambda^\bot)\ \Leftrightarrow\ \{0\}=\Phi(\Lambda)\cap \Lambda^\bot=\mathrm{Span}\{\ell_1,\dots, \ell_k\}\cap \Lambda^\bot
    $$
    and we are done.
\end{proof}
This elementary observation extremely simplifies matters, as now the problem of finding $\Lambda$ for $(*)$ to hold is equivalent to the problem of finding $V\in G_{n-k}({\Rz}^n)$ such that $V\cap \mathrm{Span}\{\ell_1,\dots, \ell_k\}=\{0\}$. Now, however, we have to to slightly more detailed with the induction.

Write ${\Rz}^n=H\oplus H^\bot$ so that $\ell_i=\ell_i'+\ell_i^-\in H\oplus H^\bot$ and for $x\in H$, $\ell_i(x)=\ell_i'(x)$. Therefore, the induction hypothesis applied to the epimorphisms $\tilde{\ell}':=(\ell_2',\dots,\ell_{k}')|_H$ gives $V\in G_{n-k}(H)$ (notice the right dimension!) such that $V\cap\mathrm{Span}\{\ell_2',\dots, \ell_k'\}=\{0\}$. 

Now, take $v\in V\cap\mathrm{Span}\{\ell_2,\dots, \ell_k\}$ and write it as $\sum_{i=2}^k\lambda_i(\ell_i'+\ell_i^-)$. Since $v\in H$, we have $\sum_{i=2}^k\lambda_i\ell_i^-=0$, whence $v\in \mathrm{Span}\{\ell_2',\dots, \ell_k'\}$ and so $v=0$ by the choice of $V$.

Finally, we claim that $V\cap \mathrm{Span}\{\ell_1,\dots, \ell_k\}=\{0\}$. Take $x\in V\cap \mathrm{Span}\{\ell_1,\dots, \ell_k\}$ and write it as $\lambda_1\ell_1+\tilde{x}$ where $\tilde{x}\in \mathrm{Span}\{\ell_2,\dots, \ell_k\}$. Recall that ${\Rz}\ell_1\cap H=\{0\}$, hence, since $x\in V\subset H$, we conclude that $\lambda_1=0$ and we are done.
%as earlier we get in fact $x=\lambda_1\ell_1'+x'$ with $x'\in \mathrm{Span}\{\ell_2',\dots, \ell_k'\}$.

This ends the proof of Theorem \ref{Glowne}. \end{proof} 
The result above shows that both claims of Yomdin are indeed true:
\begin{cor}
Theorem \ref{LIFT} and Theorem \ref{YS} are both true.
\end{cor}
\begin{proof}
    Theorem \ref{YS} follows from Theorem \ref{LIFT} as shown in \cite{Y}. Theorem \ref{YS} is a consequence of Theorem \ref{Glowne} and Theorem \ref{main}.
\end{proof}
Another consequence is that the equivalence between Yomdin's LIFT and Clarke's Lipschitz Inverse Function Theorem does hold indeed:
\begin{cor}
    Theorem \ref{LIFT} is equivalent to Theorem \ref{CLIFT}.
\end{cor}
\begin{proof}
Theorem\ref{CLIFT} implies Theorem \ref{main} (see its proof) and the latter is equivalent to Theorem \ref{LIFT} by Theorem \ref{Glowne}. For the converse implication, in order to prove Theorem \ref{CLIFT}
consider the Lipschitz map germ $F$ with $\partial F(0)$ consisting only of isomorphisms and put $G(x,y):=y-F(x)$. We are dealing now with a Lipschitz map germ $G\colon (\Rz^n\times\Rz^n,0)\to({\Rz}^n,0)$. It is easy to see that any $L\in\partial G(0,0)$ is of the form $L(x,y)=y-\ell(x)$ for some $\ell\in \partial F(0)$ whence $\mathrm{Ker} L=\mathrm{graph}\ell$ and so $\mathrm{rk} L=n$ is the maximal possible. Moreover, $\Lambda:=\{0\}\times{\Rz}^n$ is a common complement to all the kernels so that $(*)$ is satisfied for $G$. By Theorem \ref{main} and its proof, $(G^{-1}(0),0)$ is a Lipschitz graph $y=g(x)$ (in these precise variables!). Clearly, $g(x)=F(x)$ and the proof is accomplished.
\end{proof}

\section{Local properties of Lipschitz functions satisfying condition $(Y)$.}

We complete our result by stating clearly a straightforward consequence of Yomdin's condition $(Y)$.

\begin{thm}
    Assume that $f\colon ({\Rz}^n,a)\to({\Rz}^k,b)$ is a Lipschitz germ such that the Yomdin condition is satisfied at the point $a$: $$\forall\ell\in\partial f(a),\ \mathrm{rk}\ell=k.\leqno{(Y)}$$
    Then there exists a neighbourhood $U$ of $a$ such that \begin{enumerate}
        \item Condition $(Y)$ holds at any point $x\in U$;
        \item There is a bi-Lipschitz map $\Phi\colon U\to H\times V$ onto an open cube (product of intervals) $H\times V\subset {\Rz}^k\times{\Rz}^{n-k}$, for which $f|_U=\pi_k\circ \Phi$ where $\pi_k\colon {\Rz}^k\times{\Rz}^{n-k}\to{\Rz}^k$ is the natural projection onto the first $k$ coordinates (in particular $f$ is an open map);
        \item The set $H\times f^{-1}(b)\subset{\Rz}^k\times{\Rz}^n$ is a Lipschitz submanifold and there is a bi-Lipschitz map $h\colon U\to H\times f^{-1}(b)$ such that $f|_U=\pi\circ h$ where $\pi\colon H\times f^{-1}(b)\to H$ is the natural projection (i.e. $f|_U$ is a trivial Lipschitz fibration over the open set $H=f(U)$).
    \end{enumerate}
\end{thm}
\begin{proof}
    %Without loss of generality assume $a=0$ and $b=0$. 
    As observed already in \cite{C}, the subdifferential is outer semi-continuous in the Kuratowski sense: $\limsup_{z\to a}\partial f(z)\subset\partial f(a)$ where the upper limit is the set consisting of all the possible limits $\ell$ of convergent sequences $\ell_\nu\in\partial f(z_\nu)$ over all sequences $z_\nu$ tending to $0$. Obviously, $\mathrm{rk}\ell=k$ implies necessarily that $\mathrm{rk}\ell_\nu=k$ for all indices large enough. From this we get (1).
    
    %Next, by Theorem \ref{LIFT}, there is a neighbourhood $U$ of $0$ inside which $f^{-1}(0)$ is the graph of a Lipschitz function $x=\varphi(y)$ in $n-k$ variables, say, the last ones after a linear change of coordinates in ${\Rz}^n={\Rz}^{k}_x\times {\Rz}^{n-k}_y$. Then after shrinking $U$, if necessary, $\Phi(x,y)=x-\varphi(y)$ is clearly bi-Lipschitz onto an open set that we may assume to be an open cube in ${\Rz}^n$. 

    Next, since by Theorem \ref{Glowne}, condition $(Y)$ is equivalent to condition $(*)$, let us take the common complement $\Lambda$ from the latter condition. Write $\Rz^n=\Lambda\oplus\Lambda^\bot$ and consider the projection $p$ onto $\Lambda^\bot$ and an isomorphism $\lambda$ sending the latter space onto $\Rz^{n-k}$. Put $\Phi:=(f,\lambda\circ p)$ and observe that $\Phi(a)=(b,c)$ for some $c\in{\Rz}^{n-k}$. Moreover, any $L\in \partial \Phi(a)$ is of the form $(\ell,\lambda\circ p)$ for some $\ell\in \partial f(a)$ and thus is an isomorphism thanks to $(*)$. By the Clarke Inverse Function Theorem from \cite{C2}, there is a neighbourhood $U$ of $a$ on which $\Phi$ is bi-Lipschitz onto an open set that we may assume to be a cube. Then (2) follows. Note that $\Phi$ straightens up the fibres of $f|_U$.

    Finally, once we have defined $\Phi$ as in (2) above, we produce $h$ by putting for $x\in U$,  $h(x):=(f(x), \Phi^{-1}(\pi_{n-k}(\Phi(x)))$ where $\pi_{n-k}\colon H\times V\to V$ is the natural projection. Clearly, $h$ is a Lipschitz function factorising $f$ with $\pi$ and we can write down its inverse: for $(y,z)\in H\times f^{-1}(b)$, $h^{-1}(y,z)=\Pi^{-1}(y,\pi_{n-k}(\Phi(z)))$ which also is Lipschitz as required. This ends the proof of (3).
\end{proof}
\begin{rem}
By (1) in the last Theorem, $(*)$ is an open condition. Note also, that the semi-continuity used in the proof of $(1)$ implies also that the set of critical points (i.e. points $x$ for which $0\in \partial f(x)$) is closed.
\end{rem}

\section{Acknowledgements}
The author is grateful to Lev Birbrair for attracting his attention to the problem and for his interest in the subject and to Adam Bia\l o\.zyt for interesting remarks on the problem. My warmest thanks are due also to Dirk Siersma for long discussions on the subject and especially for having acted as a catalysator for my looking at the LIFT as a Theorem possible to be proved.

During the preparation of the very first version of this note (the preprint that started it all), the author was partially supported by Polish Ministry of Science and Higher Education grant  IP2011 009571.


\begin{thebibliography}{BDS}
\bibitem{BD} L. Birbrair, M. P. Denkowski, {\it Medial axis and singularities}, J. Geom. Anal. vol. 27 no.3 (2017), 2339-2380;

\bibitem{BS} L. Birbrair, D. Siersma, \textit{Metric properties of conflict sets}, Houston J. Math. 35 (1) (2009), 73-80;

\bibitem{C} F. Clarke, {\it Generalized gradients and applications}, Trans. Amer. Math. Soc. 205 (1975), 247-262;

\bibitem{C2} F. Clarke, {\it On the inverse function theorem}, Pacific J. Math. 64 no. 1 (1976), 97-102;

\bibitem{Clarke Nonsmooth} F. Clarke, {\it Optimization and Nonsmooth Analysis}, Wiley 1983; 

\bibitem{D} M. P. Denkowski, \textit{On the points realizing the distance to a definable set}, J. Math. Anal. Appl. 378  (2011), 592-602;

\bibitem{Dp} M. P. Denkowski, {\it On Yomdin's version of a Lipschitz Implicit Function Theorem},  arXiv:1610.07905v2 (2016);

\bibitem{F} D. H. Fremlin, {\it Skeletons and central sets}, Bull. London Math. Soc. (3) 74 (1997), 701-720;

\bibitem{Y} Y. Yomdin, \textit{On the local structure of a generic central set}, Comp. Math. 43 no. 2 (1981), 225-238.
\end{thebibliography}
\end{document}